\documentclass{amsart}

\usepackage{amsmath}
\usepackage{amssymb}
\usepackage{amsthm}
\usepackage{bbm}
\usepackage{graphicx}
\usepackage{color}
\usepackage{epsfig}

\newtheorem{theorem}{Theorem}
\newtheorem{prop}[theorem]{Proposition}
\newtheorem{lemma}[theorem]{Lemma}
\newtheorem{cor}[theorem]{Corollary}

\newtheorem{question}{Question}

\newtheorem{conj}[question]{Conjecture}

\newcommand{\floor}[1]{\ensuremath{\left \lfloor {#1} \right \rfloor}}

\newcommand{\ang}[1]{\ensuremath{\left \langle {#1} \right \rangle}}

\newcommand{\Z}{{\mathbb Z}}

\newcommand{\bX}{{\bf X}}

\newcommand{\sym}{\textsc{sym}}
\newcommand{\dv}{\textsc{div}}
\newcommand{\alldiv}{\textsc{alldiv}}

\allowdisplaybreaks

\title{Symmetric and Asymptotically Symmetric Permutations}
\author{Joshua N. Cooper and Andrew Petrarca}
\date{\today}

\begin{document}

\begin{abstract} We consider two related problems arising from a question of R.~Graham on quasirandom phenomena in permutation patterns.  A ``pattern'' in a permutation $\sigma$ is the order type of the restriction of $\sigma : [n] \rightarrow [n]$ to a subset $S \subset [n]$.  First, is it possible for the pattern counts in a permutation to be exactly equal to their expected values under a uniform distribution?  Attempts to address this question lead naturally to an interesting number theoretic problem: when does $k!$ divide $\binom{n}{k}$?  Second, if the tensor product of a permutation with large random permutations is random-like in its pattern counts, what must the pattern counts of the original permutation be?  A recursive formula is proved which uses a certain permutation ``contraction.''
\end{abstract}

\maketitle

\section{Introduction}

In a well-shuffled deck, the royal spades are equally likely to be in any particular order.  Consequently, each such ordering (for example, $\spadesuit$AKQJ or $\spadesuit$KJAQ) will occur with probability 1/24.  Likewise, any fixed subsequence of a (uniformly) random permutation is as likely to be in any order as any other: probability $1/k!$ for any such ``pattern'' of $k$ symbols.  If one considers all $\binom{n}{k}$ possible subsequences, then the number of patterns of each of the $k!$ order types will be, on average, $\binom{n}{k}/k!$.

Suppose that, for some permutation $\sigma$, it is indeed true that each of the patterns on $k$ symbols occurs approximately $\binom{n}{k}/k!$ times.  It is not hard to see that the same statement with $k$ replaced by $k-1$ is also true of $\sigma$, but is it true if $k$ is replaced with $k+1$?  Although it might seem highly unlikely, precisely this phenomenon occurs with graphs, as discovered and popularized by Chung, Graham, and Wilson in their seminal paper, ``Quasirandom Graphs'' (\cite{ChGrWi89}).  They showed that, if the number of $k$-vertex {\it subgraphs} of a graph $G$ occur at about the ``right rate'' -- i.e., the expected number of times they would occur were $G$ chosen uniformly at random from all $n$-vertex graphs -- then this same statement must also hold for the $(k+1)$-vertex subgraphs {\it iff $k \geq 4$}.

Graham asked (\cite{Co04,Gr01}), does this upward-implication also occur for permutations?  (The exact definitions are a bit technical, and therefore reserved for the next section.)  That it cannot hold for $k=1$ is obvious.  The problem of $k=2$ requires a moment's thought to see that it is also not possible.  In fact, $k=3$ was also resolved in the negative early on by Chung (\cite{Ch01}).  We conjecture that this pattern continues for all $k$.  Unfortunately, even $k=4$ is still open.  In the present manuscript, we offer a construction, which we call ``inflatable'' permutations, that may provide some inroads into the problem.  Already, it has provided millions of proofs in the case $k=3$.  We discuss this construction and some of its consequences in Section \ref{sec:asymptotic}.  In particular, we offer a formula for computing the pattern statistics of permutations whose ``inflation'' by a random permutation has random-like $k$-pattern counts.

A related question to the existence of inflatable permutations, and one which might also shed light on Graham's conjecture, in that of {\it exact} achievement of the expected value.  Is it possible for the number of length $k$ patterns of each type to be exactly $\binom{n}{k}/k!$ ?  The case $k=1$ is trivial, while $k=2$ again requires a bit of thought, and $k=3$ is far from obvious.  We show that there do exist permutations satisfying this condition when $k=3$.  Again, $k=4$ remains open.  And, again, this problem very closely parallels one in graph theory: can all the $k$-vertex subgraph counts be {\it exactly} equal?  Janson and Spencer consider this problem in \cite{JaSp92} and, intriguingly, are also able to resolve the question (in the positive) only for $k=1$, $2$, and $3$.

Certainly, a necessary condition is that the quantity $\binom{n}{k}/k!$ be an integer.  Hence, we consider the problem of identifying the integers for which this can happen, since these $n$ are the natural targets for a computer search.  We show that the divisibility condition depends only on the congruence class of $n$ modulo some integer $T$, and we determine the smallest such $T$.  In the case $k=4$, it turns out that the least $n$ where exact achievement of the expected values {\it might} happen (other than the trivial case when $n < k$) is $n = 64$.  Despite thousands of processor-hours, finding the desired permutation in the haystack of $64! \approx 10^{89}$ possibilities has remained elusive.  We discuss this question in Section \ref{sec:exact}.

\section{Preliminaries}

Write $\mathfrak{S}_n$ for the symmetric group of order $n$.  Let $\bX^\tau(\sigma)$, for two permutations $\tau \in \mathfrak{S}_k$ and $\sigma \in \mathfrak{S}_n$, denote the number of copies of $\tau$ that appear in $\sigma$ as a pattern.  That is, $\bX^{\tau}(\sigma)$ is the number of elements $S$ of $\binom{[n]}{k}$ so that $\tau(i) < \tau(j)$ iff $\sigma(i) < \sigma(j)$ for $i,j \in S$.  We call a permutation $\sigma$ ``$k$-symmetric'' and write $\sym_k(\sigma)$ if, for each $\tau, \tau^\prime \in \mathfrak{S}_k$,
$$
\bX^{\tau}(\sigma) = \bX^{\tau^\prime}(\sigma).
$$
Clearly, it is equivalent to require that $\bX^{\tau}(\sigma) = \binom{n}{k}/k!$ for each $\tau \in \mathfrak{S}_k$.  We will say that a $k$-symmetric permutation $\sigma$ is {\it trivially} $k$-symmetric if $\binom{n}{k} = 0$, and {\it nontrivially} otherwise.  Hence, it is necessary that $k! | \binom{n}{k}$, under which circumstances we will write $\dv_k(n)$.  Furthermore, since it is easy to see that $\sym_k(\sigma) \Rightarrow \sym_{k-1}(\sigma)$ for $k \geq 2$ (see, e.g., \cite{Co04}), the condition
$$
\alldiv_k(n) = \dv_1(n) \wedge \cdots \wedge \dv_k(n)
$$
is also necessary in order for there to exist a permutation $\sigma \in \mathfrak{S}_n$ which is $k$-symmetric.  We have the following conjecture.

\begin{conj} \label{conj:modularcondition} If $\alldiv_k(n)$, then there exists a $k$-symmetric permutation on $n$ symbols.
\end{conj}

We are also interested in the least $n \geq k$ so that $\alldiv_k(n)$, since this at least gives one a starting point in a computer search for nontrivially $k$-symmetric permutations.  Call this quantity $f(k)$.

\section{Divisibility Conditions for $k$-Symmetry} \label{sec:exact}

The condition $\dv_k(n)$ is equivalent to
$$
k!^2 | n^{\underline{k}} = n(n-1)(n-2)\cdots(n-k+1).
$$
Hence, if we factor $k! = \prod_{p} p^{e_p(k)}$, we need to ensure that, for each prime $p$,
$$
\nu_p(n(n-1)(n-2)\cdots(n-k+1)) \geq 2e_p(k),
$$
where $\nu_p(m)$ denotes the the greatest $r$ so that $p^r | m$, a statement which is sometimes written $p^r \| m$.  (We take $\nu_p(0) = \infty$ for any $p$.)  Products of consecutive integers have a long history in the literature and were of particular interest to Erd\H os, who considered their prime factorizations and famously showed (with Selfridge) that they could never be perfect powers.  (See, for example, \cite{Er39,Er85,Gy99,Le65}.  The manuscripts \cite{ErSt77} and \cite{KoRe02} come closest to the present questions.)

We begin with the following simple lemma.

\begin{lemma} \label{minproperty} For any $s,t$, if $\nu_p(s) < \nu_p(t)$, then $\nu_p(t + s) = \nu_p(s)$.
\end{lemma}
\begin{proof} If $t = 0$ (or $t + s = 0$, in which case the statement is vacuous), this is clear.  Suppose $t = kp^a$ and $t+s = lp^b$ with $b = \nu_p(t+s)$, $a$ and $b$ finite.  Then
$$
s = lp^b - kp^a = p^{\min\{a,b\}} (l p^{b - \min\{a,b\}} - k p^{a - \min\{a,b\}}),
$$
whence $\nu_p(s) \geq \min\{a,b\}$.  Since $\nu_p(s) < a$, this implies $\nu_p(s) = b$.
\end{proof}

Since $p^{\nu_p(\cdot)}$ is completely multiplicative,
\begin{equation} \label{eq1}
\nu_p(n(n-1)(n-2)\cdots(n-k+1)) = \sum_{j=0}^{k-1} \nu_p(n-j).
\end{equation}
Define $f_{p,k}(n)$ to be the right-hand side of this expression.  If $n - j$ is a multiple of $p^{2e_p(k)}$ for any $0 \leq j \leq k-1$, then $f_{p,k} \geq 2e_p(k)$.  If not, then, by Lemma \ref{minproperty}, $f_{p,k}(n) \geq f_{p,k}(n_0)$, where $n_0$ is the least nonnegative representative of $n$ modulo $p^{2e_p(k)}$.  Either way, the fact that $f_{p,k}(n) \geq 2e_p(k)$ depends only on the congruence class to which $n$ belongs modulo $p^{2e_p(k)}$.  Furthermore, one need only ask for a divisor of the modulus $p^{2e_p(k)}$: let us denote by $a_p(k)$ the smallest $m$ so that $f_{p,k}(n) \geq 2e_p(k)$ depends only on the congruence class of $n$ modulo $p^m$.  (We set $a_p(k) = 0$ if this modulus does not exist, i.e., $p > k$.)  Write $T(k)$ for the minimum integer $m^\prime$ so that $\alldiv_k(n)$ depends only on the congruence class of $n$ modulo $m^\prime$.  Clearly,
$$
T(k) = \prod_{p} p^{a_p(k)},
$$
providing us with a necessary condition for the existence of $k$-symmetric permutations.  For example, when $k = 3$, we have $T(k) = 36$, and the relevant congruence classes are $0$, $1$, $9$, $20$, $28$, and $29$.  Hence, the first $n$ so that a nontrivially $3$-symmetric permutation may exist is $n=9$.  In fact, there are exactly two such permutations -- $349852167$ and $761258943$ -- which happen to be each other's inverse, reverse, and conjugate.

\begin{lemma} \label{notsmallerthank} For any $k$ and prime $p \leq k$, $p^{a_p(k)} \geq k$.
\end{lemma}
\begin{proof} Suppose $p^{a_p(k)} < k$.  Then, for any $n$, if $n_0$ is the least nonnegative representative of $n$ modulo $p^{a_p(k)}$, $f_{p,k}(n) \geq 2e_p(k)$ iff $f_{p,k}(n_0) \geq 2e_p(k)$.  Hence,
$$
f_{p,k}(n_0) = \sum_{j=0}^{k-1} \nu_p(n_0-j) = \infty,
$$
since $j = n_0 < k$ satisfies $\nu_p(n_0 - j) = \infty$.  So, $f_{p,k}(n)$ holds for all $n$.  But
$$
f_{p,k}(k) = \sum_{j=1}^k \nu_p(j) = e_p(k) < 2e_p(k),
$$
unless $e_p(k) = 0$, i.e., $k < p$, a contradiction.
\end{proof}

\begin{lemma} If $n-p^{a_p(k)} \in \{0,\ldots,k-1\}$, then $f_{p,k}(n) = 2e_p(k)$.
\end{lemma}
\begin{proof} First, suppose that $f_{p,k}(n) < 2e_p(k)$ for some $n$ with $n-j = p^{a_p(k)}$, $0 \leq j \leq k-1$.  Then $n^\prime = p^{2e_p(k)} + j$ satisfies $\nu_p(n^\prime - j) = 2e_p(k)$, whence $f_{p,k}(n^\prime) \geq 2e_p(k)$.  However, $n^\prime \equiv n \pmod{p^{a_p(k)}}$, contradicting the definition of $a_p(k)$.

Now, suppose that $f_{p,k}(n) \geq 2e_p(k) + 1$ for some $n$ with $n-j = p^{a_p(k)}$  Let $n^\prime = n + sp^{a_p(k)-1}$ for some $s \in \Z$.  Let $i \neq j$ with $0 \leq i \leq k-1$.  Then, since $|i-j| = |(n-j)-(n-i)| < k$, Lemma \ref{notsmallerthank} implies $|i-j| < p^{a_p(k)}$, so that $\nu_p(i-j) < a_p(k)$.  By Lemma \ref{minproperty}, this in turn gives that
$$
\nu_p(n-i) = \nu_p(n-j - (i-j)) = \nu_p(i-j) < a_p(k).
$$
Therefore, we have
\begin{align*}
\nu_p(n^\prime - i) &= \nu_p(p^{a_p(k)} + sp^{a_p(k)-1} + (j - i)) \\
& \geq \min\{a_p(k)-1,\nu_p(j-i)\} \\
& \geq \nu_p(j-i) = \nu_p(n-i).
\end{align*}
This yields
\begin{align*}
f_{p,k}(n^\prime) &= \sum_{i=0}^{k-1} \nu_p(n^\prime - i )\\
&\geq \nu_p(n^\prime - j) + \sum_{i \neq j} \nu_p(n - i) \\
&= \nu_p(p^{a_p(k)} + sp^{a_p(k)-1}) + \sum_{i \neq j} \nu_p(n - i) \\
&\geq a_p(k) - 1 + \sum_{i \neq j} \nu_p(n - i) = f_{p,k}(n) - 1 = 2e_p(k).
\end{align*}
On the other hand, suppose $\nu_p(n-j) < a_p(k)$ for all $0 \leq j \leq k-1$.  If there exists a $j$ so that $\nu_p(n-j) = a_p(k)-1$, i.e., $n-j = sp^{a_p(k)-1}$ for some $s$, then $n^\prime = n + p^{a_p(k)-1} (p-s)$, so $f_{p,k}(n) \geq 2e_p(k)$ by the above argument.  Hence, we can assume that $\nu_p(n-j) \leq a_p(k) - 2$ for all $0 \leq j \leq k-1$.  Setting $n^\prime = n + sp^{a_p(k)-1}$ gives
\begin{align*}
f_{p,k}(n^\prime) &= \sum_{i=0}^{k-1} \nu_p(n^\prime - i )\\
&= \sum_{i=0}^{k-1} \nu_p(n - i + sp^{a_p(k)-1})\\
&= \sum_{i=0}^{k-1} \nu_p(n - i)\\
&= f_{p,k}(n^\prime) < 2e_p(k).
\end{align*}
Therefore, whether $f_{p,k}(n) \geq 2e_p(k)$ or not depends only on the congruence class of $n$ modulo $p^{a_p(k)-1}$, contradicting the definition of $a_p(k)$.

\end{proof}

Suppose that $\nu_p(n-j) = a_p(k)$.  Then
\begin{align*}
a_p(k) &= \max_{0 \leq j \leq k-1} \left ( 2e_p(k) - \sum_{\substack{0 \leq i \leq k-1 \\ i \neq j}} \nu_p(n-i) \right ) \\
&= 2e_p(k) - \min_{0 \leq j \leq k-1} \sum_{\substack{0 \leq i \leq k-1 \\ i \neq j}} \nu_p(n-i)\\
&= 2e_p(k) - \min_{0 \leq j \leq k-1} \left ( \sum_{-k+1+j \leq s \leq -1} \nu_p(n-j+s) + \sum_{1 \leq s \leq j} \nu_p(n-j+s) \right )\\
\end{align*}
Since $|s| < k$ implies $|s| < p^{a_p(k)}$ implies $\nu_p(s) < a_p(k)$, we can apply Lemma \ref{minproperty} to this expression, obtaining
\begin{align*}
a_p(k) &= 2e_p(k) - \min_{0 \leq j \leq k-1} \left ( \sum_{s=1}^{k-1-j} \nu_p(s) + \sum_{s=1}^j \nu_p(s) \right ) \\
&= 2e_p(k) - \min_{0 \leq j \leq k-1} \left ( e_p(k-1-j) + e_p(j) \right ) \\
&= 2e_p(k) - e_p(k-1) + \max_{0 \leq j \leq k-1} \nu_p\left ( \binom{k-1}{j} \right ) \\
\end{align*}
We may now apply Kummer's Theorem on multinomial coefficients:

\begin{theorem}[Kummer's Theorem] $\nu_p(\binom{x}{y})$ is given by the number of integers $j > 0$ for which $\{y/p^j\} > \{x/p^j\}$.
\end{theorem}

The inequality
$$
\{y/p^j\} > \{x/p^j\}
$$
holds whenever the integer obtained from the last $j$ $p$-ary digits of $y$ are greater than the corresponding integer for $x$.  By setting each digit of $y$ to $p-1$, while keeping $y < x$, the number of such $j$ is maximized.  That is, the $y$ which maximizes $\nu_p(\binom{x}{y})$ is
$$
y = \sum_{j=0}^{\floor{\log_p(x)}-1} (p-1) p^j,
$$
whence $\max_j \nu_p(\binom{k-1}{j})$ is the number of nonleading base-$p$ digits of $k-1$ to the left (inclusive) of the least significant $p$-ary digit which is {\it not} $p-1$.  This is the same as one less than the number of $p$-ary digits of $k-1$ minus the number of times that $p$ divides $(k-1)+1 = k$, or
$$
\floor{\log_p(k-1)} - \nu_p(k),
$$
unless $k$ is a power of $p$, in which case it is $0$.  We may then conclude the following.

\begin{theorem} For $p \leq k$,
$$
a_p(k) = e_p(k) + \floor{\log_p k} .
$$
\end{theorem}
\begin{proof} Applying the above equalities when $k$ is not an integer power of $p$,
\begin{align*}
a_p(k) &= 2e_p(k) - e_p(k-1) + \max_{0 \leq j \leq k-1} \nu_p\left ( \binom{k-1}{j} \right ) \\
&= e_p(k) + \left [ e_p(k) - e_p(k-1) \right ] + \left ( \floor{\log_p(k-1)} - \nu_p(k) \right ) \\
&= e_p(k) + \nu_p(k) + \floor{\log_p(k-1)} - \nu_p(k) \\
&= e_p(k) + \floor{\log_p(k-1)} = e_p(k) + \floor{\log_p(k)}.
\end{align*}
When $k = p^r$,
\begin{align*}
a_p(k) &= 2e_p(k) - e_p(k-1) + \max_{0 \leq j \leq k-1} \nu_p\left ( \binom{k-1}{j} \right ) \\
&= e_p(k) + \left [ e_p(k) - e_p(k-1) \right ] + 0 \\
&= e_p(k) + \nu_p(k) = e_p(k) + r = e_p(k) + \floor{\log_p(k-1)}.
\end{align*}
\end{proof}

\begin{cor} For any $k \geq 2$ and $p \leq k$,
$$
-2 < a_p(k) - \frac{k}{p-1} \leq \log_p k.
$$
\end{cor}
\begin{proof}
We employ the well-known formula
$$
e_p(t) = \sum_{j=1}^\infty \floor{\frac{t}{p^j}}.
$$
and the consequent estimate
$$
\frac{t (1-p^{-\floor{\log_p t}})}{p-1} - \floor{\log_p t} < e_p(t) \leq \frac{t (1-p^{-\floor{\log_p t}})}{p-1} \leq \frac{t}{p-1}.
$$
This implies that
$$
a_p(k) \leq \frac{k}{p-1} + \log_p k.
$$
In the other direction,
\begin{align*}
a_p(k) &> \frac{k (1-p^{-\floor{\log_p k}})}{p-1} - \floor{\log_p k}  + \floor{\log_p k} \\
& \geq \frac{k - kp^{-\log_p k + 1}}{p-1} \\
&= \frac{k - p}{p-1} \geq \frac{k}{p-1} - 2.
\end{align*}
\end{proof}

It remains to describe the size of $\log m^\prime(k) = \sum_{p} a_p(k) \log p$.
\begin{cor} $\log m^{\prime} = k \log k + O(k)$.
\end{cor}
\begin{proof}
\begin{align*}
\sum_{p} a_p(k) \log p &= \sum_{p \leq k} \frac{k \log p}{p-1} + O \left ( \sum_{p \leq k} \log p (\log_p(k) + 1) \right ) \\
&= k \sum_{p \leq k} \frac{\log p}{p-1} + O \left ( \sum_{p \leq k} (\log k + \log p) \right )\\
&= k \log k + O \left (k + \log k \cdot \frac{k}{\log k} \right ) \\
&= k \log k + O \left (k \right ).
\end{align*}
\end{proof}

\section{Asymptotic Symmetry and Inflatability} \label{sec:asymptotic}

In this section, we investigate the possibility of using a ``blow-up'' construction to address Graham's conjecture.  First, define a sequence of permutations $\sigma_i \in \mathfrak{S}_{n_i}$, $n_i \rightarrow \infty$, to be ``asymptotically $k$-symmetric'' if, for each $\tau \in \mathfrak{S}_k$,
$$
\lim_{i \rightarrow \infty} \frac{k!^2 \bX^{\tau}(\sigma_i)}{n^k} = 1.
$$
Customarily dropping indices, we may state asymptotic $k$-perfection as the condition that $\bX^{\tau}(\sigma) = \binom{n}{k} (1/k! + o(1))$ for each $\tau \in \mathfrak{S}_k$.  We write ${\sym}^\prime_k(\sigma)$ to mean that the sequence $\sigma_i$ is asymptotically $k$-symmetric.

\begin{question}[Graham] Does there exist a $k \geq 1$ so that
$$
{\sym}^\prime_k(\sigma) \Rightarrow {\sym}^\prime_{k+1}(\sigma)?
$$
\end{question}

We define a tensor product $\otimes$ of permutations, corresponding to the Kronecker product -- a.k.a. tensor/outer product -- of the corresponding permutation matrices.  (This product has appeared before in the literature, and appears in a surprising number of applications: see for example \cite{EgPuBe97}.)  Suppose $\pi_1 \in \mathfrak{S}_a$ and $\pi_2 \in \mathfrak{S}_b$.  If we consider elements of $\mathfrak{S}_n$ to be functions of $\{0,\ldots,n-1\}$, then $\pi_1 \otimes \pi_2 \in \mathfrak{S}_{ab}$ is defined by
$$
(\pi_1 \otimes \pi_2)(j) = b \pi_1 \left (\floor{\frac{j}{b}}\right) + \pi_2(j \!\!\!\! \pmod{b})
$$
for $0 \leq j \leq ab-1$.  Although this notation provides some simplicity, we will usually consider the elements of $\mathfrak{S}_n$ to be functions of $[n] = \{1,\ldots,n\}$ in order to agree with the nearly universal convention in literature on permutations.  

Let $\rho_n$ be a uniformly distributed random permutation on $n$ symbols.  We believe it very likely that permutations $\sigma$ exist so that $\sigma \otimes \rho_n$ is asymptotically $k$-symmetric but not asymptotically $(k+1)$-symmetric.  Evidently, the existence of such permutations would resolve Graham's conjecture in the negative.  We begin our description of these $\sigma$ with $k=2$.

\begin{prop} \label{prop2inflatable} For $\sigma \in \mathfrak{S}_m$, $m$ fixed, if the permutation $\sigma \otimes \rho_n$ is asymptotically $2$-symmetric, then $\sigma$ is $2$-symmetric.
\end{prop}
\begin{proof} Let $\pi_n = \sigma \otimes \rho_n$.   Suppose that ${\sym}^\prime_{2}(\pi_n)$, i.e.,
$$
\bX^{(12)}(\pi_n) = \bX^{(21)}(\pi_n) = (1+ o(1)) (mn)^2/4.
$$
Each of the pairs of indices in $[mn]$ either appear in different blocks of the form $B_k = [(k-1)n+1,kn]$, $k \in [m]$, or else in the same block.  The number of coinversions (and hence inversions, as well) of the latter type is, w.h.p.,
$$
\frac{m n^2}{4} (1+ o(1))
$$
while the number of coinversions of the former type is exactly
$$
\bX^{(12)}(\sigma) n^2.
$$
Therefore, w.h.p.,
$$
\frac{m^2 n^2}{4} (1+ o(1)) = \frac{m n^2}{4} (1+ o(1)) + \bX^{(12)}(\sigma) n^2.
$$
Dividing by $n^2/4$ yields
$$
4 \bX^{(12)}(\sigma) = (1+ o(1)) m(m-1).
$$
However, the left hand side does not depend on $n$, so we can take $n \rightarrow \infty$ in the above expression and get that $\bX^{(12)}(\sigma) = m(m-1)/4$, i.e., $\sigma$ is $2$-symmetric.
\end{proof}

\begin{prop} \label{prop3inflatable} Suppose $\pi_n = \sigma \otimes \rho_n$ is asymptotically $3$-symmetric for some fixed $\sigma \in \mathfrak{S}_m$.  Then $m \equiv 0,1, 9, \textrm{ or } 32 \pmod{36}$.
\end{prop}
\begin{proof} Triples of indices in $\binom{[m]}{3}$ come in four varieties: (a) all indices are in the same block $B_k$, (b) two indices come from $B_i$ and one index from $B_j$ with $i<j$, (c) two index comes from $B_i$ and two indices from $B_j$ with $i<j$, and (d) all three indices appear in different blocks.  The number of each pattern $\tau \in \mathfrak{S}_3$ that appear either in form (a) or (d) is, w.h.p.,
$$
m n^3 (1/36 + o(1)) + \bX^{\tau}(\sigma) n^3,
$$
a quantity which we will call $M$.  Counting the rest of the patterns, then, yields
\begin{align*}
\bX^{(123)}(\pi_n) & = M + \bX^{(12)}(\sigma) n^3 (1/2+o(1)) + \bX^{(12)}(\sigma) n^3 (1/2+o(1)) \\
& = M + m(m-1)n^3 (1/4+ o(1)) \\
\bX^{(132)}(\pi_n) & = M + \bX^{(12)}(\sigma) n^3 (1/2+o(1)) = M + m(m-1)n^3 (1/8 + o(1)) \\
\bX^{(213)}(\pi_n) & = M + \bX^{(12)}(\sigma) n^3 (1/2+o(1)) = M + m(m-1)n^3 (1/8 + o(1)) \\
\bX^{(231)}(\pi_n) & = M + \bX^{(21)}(\sigma) n^3 (1/2+o(1)) = M + m(m-1)n^3 (1/8 + o(1)) \\
\bX^{(312)}(\pi_n) & = M + \bX^{(21)}(\sigma) n^3 (1/2+o(1)) = M + m(m-1)n^3 (1/8 + o(1)) \\
\bX^{(321)}(\pi_n) & = M + \bX^{(21)}(\sigma) n^3 (1/2+o(1)) + \bX^{(21)}(\sigma) n^3 (1/4+o(1)) \\
& = M + m(m-1)n^3 (1/4 + o(1))
\end{align*}
where the last equality on each line comes from the fact that $\sigma$ is $2$-symmetric (owing to the previous Proposition and the fact that ${\sym}^\prime_3(\pi_n) \rightarrow {\sym}^\prime_2(\pi_n)$.)  Hence, we have
$$
n^{-3} \bX^{\tau}(\pi_n) = o(1) + m/36 + \bX^{\tau}(\sigma) + \binom{m}{2} \cdot \left \{ \begin{array}{ll} 1/2 & \textrm{if $\tau$ is monotone} \\ 1/4 & \textrm{otherwise} \end{array} \right . .
$$
Since $\pi_n$ is asymptotically $3$-symmetric, any two of these quantities are equal.  Letting $n \rightarrow \infty$, we have
$$
\bX^{(123)}(\sigma) + \frac{1}{2} \binom{m}{2} = \bX^{(132)}(\sigma) + \frac{1}{4} \binom{m}{2}.
$$
It is clear that this implies that $4 | \binom{m}{2}$, i.e., $8 | m(m-1)$, which is simply the condition that $m$ is $0$ or $1$ mod $8$.  Furthermore,
\begin{align*}
\binom{m}{3} & = \sum_\tau \bX^{\tau}(\sigma) \\
& = 2 \bX^{123}(\sigma) + 4 \left (\bX^{123}(\sigma) + \frac{m(m-1)}{8} \right ) \\
& = 6 \bX^{123}(\sigma) + \binom{m}{2}
\end{align*}
Since $\binom{m}{3} - \binom{m}{2} = m(m-1)(m-5)/6$ and $\bX^{123}(\sigma)$ is integral, we may conclude that $36 | m(m-1)(m-5)$.  This amounts to the condition that $m \equiv 0$, $1$, $5$, $9$, $19$, $23$, $27$, $28$, or $32 \pmod{36}$.  Since $m \equiv 0 \textrm{ or } 1 \pmod{8}$, the only possibilities are that $m \equiv 0, 1, 9, \textrm{ or } 32 \pmod{36}$.
\end{proof}

\begin{cor} If $\sigma \otimes \rho_n$ is asymptotically $3$-symmetric, then $\sigma$ is $2$-symmetric,
\begin{align*}
\bX^{(123)}(\sigma) & = \bX^{(321)}(\sigma) = \frac{m(m-1)(m-5)}{36} \qquad \textrm{and} \\
\bX^{(132)}(\sigma) & = \bX^{(213)}(\sigma) = \bX^{(231)}(\sigma) = \bX^{(312)}(\sigma) = \frac{m(2m-1)(m-1)}{72}.
\end{align*}
\end{cor}

Call a permutation so that $\sigma \otimes \rho_n$ is asymptotically $k$-symmetric ``$k$-inflatable.''  Note that, since asymptotic $(k+1)$-symmetry implies asymptotic $k$-symmetry, $(k+1)$-inflatability implies $k$-inflatability.  Furthermore, $2$-inflatability is just $2$-symmetry by Proposition \ref{prop2inflatable}, while $3$-inflatability requires the conditions given in the Corollary above.  The least $m$ so that the requisite quantities are integers is $m = 9$.  A computer search concluded that a $3$-inflatable permutation does indeed occur for the first time at $m=9$.  (In this case, $\frac{m(m-1)(m-5)}{36} = 8$ and $\frac{m(2m-1)(m-1)}{72} = 17$.)  In fact, there are exactly four permutations on $9$ symbols which are $3$-inflatable: $438951276$, $472951836$, $638159274$, and $672159834$.  Since each of these permutations has $\bX^{(1234)}(\sigma) = 0$, we may immediately conclude that $\sym^\prime_3(\sigma) \not \Rightarrow \sym^\prime_4(\sigma)$.

We have conducted searches for several larger values of $m$ where the modular restrictions hold, and in each case there were millions of $3$-inflatable permutations.  In fact, the resulting multitude of examples were found in a much smaller search space: those permutations which, like the examples given above, are the reverses and complements of their own inverses.  The inverse of a permutation is just its inverse as a function; the reverse of $\pi \in \mathfrak{S_n}$ is the permutation $\pi^\prime \in \mathfrak{S}_n$ with $\pi^\prime(j) = \pi^\prime(n+1-j)$; and the complement $\bar{\pi}$ of $\pi$ is defined by $\bar{\pi} = n+1 - \pi(j)$.  These operations are involutions on $\bigcup_n \mathfrak{S}_n$, and together they generate the group $\Gamma \cong D_8$ of maps that commute with projections onto subpatterns (i.e., the maps $\phi_I : \mathfrak{S}_n \rightarrow \mathfrak{S}_k$ for $I \subset [n]$ that take $\sigma$ to the order type of $\sigma|_I$).  Permutations $\pi$ such that $\pi^{-1} = \pi^\prime = \bar{\pi}$ have a particularly elegant representation as a pair of Young tableaux under the RSK correspondence: the underlying Ferrers diagram is symmetric about its diagonal, and the tableaux are each others' transpose.  When we pursued a search restricted to such permutations for $4$-inflatable examples at $n=64$, unfortunately, none were found, despite the use of thousands of dedicated processor-hours.  The authors have come to believe that actually are no $4$-inflatable permutations of this type.

\begin{conj} There is no $4$-inflatable permutations on $64$ symbols whose inverse is also its complement.
\end{conj}

\noindent The next result describes the conditions on pattern-counts that are necessary and sufficient for $4$-inflatability.

\begin{prop} \label{prop4inflatable} A permutation $\sigma \in \mathfrak{S}_m$ is $4$-inflatable iff
\begin{enumerate}
\item $\bX^{(1234)}(\sigma) = \bX^{(4321)}(\sigma) = \frac{m(m-1)(m^2-11m+44)}{576}$
\item $\bX^{(1243)}(\sigma) = \bX^{(2134)}(\sigma) = \bX^{(3421)}(\sigma) = \bX^{(4312)}(\sigma) = \frac{m(m-1)(m-2)(m-5)}{576}$
\item $\bX^{(1342)}(\sigma) = \bX^{(1423)}(\sigma) = \bX^{(4132)}(\sigma) = \bX^{(2431)}(\sigma) = \bX^{(3241)}(\sigma) = \\
       \bX^{(2314)}(\sigma) = \bX^{(4213)}(\sigma) = \bX^{(3124)}(\sigma) = \frac{m(m-1)(m^2-3m-1)}{576}$
\item $\bX^{(1432)}(\sigma) = \bX^{(4123)}(\sigma) = \bX^{(2341)}(\sigma) = \bX^{(3214)}(\sigma) = \frac{m(m-1)(m^2-7m+1)}{576}$
\item $\bX^{(1324)}(\sigma) = \bX^{(4231)}(\sigma) = \frac{m(m-1)(m^2-3m+13)}{576}$
\item $\bX^{(2143)}(\sigma) = \bX^{(3412)}(\sigma) = \frac{m(m-1)(m^2-7m-4)}{576}$
\item $\bX^{(2413)}(\sigma) = \bX^{(3142)}(\sigma) = \frac{m(m-1)(m^2+m+1)}{576}$
\end{enumerate}
\end{prop}

Note that the pattern count $\bX^\tau$ depends only on the orbit under $\Gamma$ to which $\tau$ belongs, so we need only check seven conditions out of the possible 24, i.e., one from each of the lists above.  The first $m \geq 4$ for which these conditions yield integers is $m=64$, when they are, respectively, $24052$, $25606$, $27321$, $25543$, $27419$, $25508$, and $29127$.  Note that
$$
\ang{24052, 25606, 27321, 25543, 27419, 25508, 29127} \cdot \ang{2,4,8,4,2,2,2} = 635376,
$$
which is just $\binom{64}{4}$.

Proposition \ref{prop4inflatable} will be a consequence of the following Lemma, which generalizes some of the above arguments.  First, we introduce some notation. For $\tau \in \mathfrak{S}_k$, let $\Pi^{\tau}_t$ denote the partitions of $[k]$ into $t$ intervals which are {\it consistent with} $\tau$.  That is, each element of $\Pi^{\tau}_t$ has the form $\{J_1 < \cdots < J_t\}$, where each $J_j$ is an interval so that $\tau(J_j)$ is an interval.  (We write $S < T$ for $S,T \subset \Z$ whenever every element of $S$ is less than every element of $T$.)  Let $I_j = [(n-1)j+1,\ldots,nj]$ for $j = 1, \ldots, m$.  The sets $(\sigma \otimes \rho_n)(I_j)$ are disjoint intervals.  Therefore, whenever $\tau$ occurs in $\sigma \otimes \rho_n$ on $\{i_1 < \cdots < i_k\} \subset [mn]$, partitioning the $i_j$ according to the $I_s$ to which they belong yields an element of $\Pi^\tau_t$ for some $t$.  More specifically, if we define $f : [k] \rightarrow [m]$ by setting $f(j) = s$ whenever $i_j \in I_s$, then the collection of sets $\pi = \{f^{-1}(s) : s \in [m]\}$ is a partition of $[k]$ into $t$ intervals consistent with $\tau$.  Since $\tau$ well-orders the elements of $\pi$, we write $\tau / \pi \in \mathfrak{S}_t$ for the ``contracted'' permutation induced by $\tau$ on the indices of the $J_j \in \pi$.  Formally, if we choose $r_1,\ldots,r_t \in [k]$ so that $f(r_j) < f(r_{j+1})$ for $1 \leq j \leq t-1$, then $(\tau / \pi)(a) < (\tau / \pi)(b)$ iff $\tau(r_a) < \tau(r_b)$.  It is not hard to see that the contraction operation is well-defined.

\begin{lemma} \label{lemmapartitionformula} For $\tau \in \mathfrak{S}_k$,
$$
\frac{m^k}{k!^2} = \sum_{t=1}^k \sum_{\pi \in \Pi^\tau_t} \bX^{\tau / \pi}(\sigma) \prod_{J \in \pi} |J|!^{-2}.
$$
\end{lemma}
\begin{proof}  Suppose that $\sigma$ is $k$-inflatable.  Since $\bX^{\tau}(\sigma \otimes \rho_n) = \binom{nm}{k}(1/k!+o(1))$, we have
\begin{align*}
\frac{m^k}{k!^2} &= \lim_{n \rightarrow \infty} \frac{\bX^{\tau}(\sigma \otimes \rho_n)}{n^k} \\
&= \lim_{n \rightarrow \infty} n^{-k} \sum_{t=1}^k \sum_{\pi \in \Pi^\tau_t} \bX^{\tau / \pi}(\sigma) \prod_{J \in \pi} \frac{1}{|J|!} \binom{n}{|J|} \\
&= \lim_{n \rightarrow \infty} n^{-k} \sum_{t=1}^k \sum_{\pi \in \Pi^\tau_t} \bX^{\tau / \pi}(\sigma) \prod_{J \in \pi} \frac{n^{|J|}(1+o(1))}{|J|!^2} \\
&= \lim_{n \rightarrow \infty} n^{-k} \sum_{t=1}^k \sum_{\pi \in \Pi^\tau_t} \bX^{\tau / \pi}(\sigma) n^{\sum_{J \in \pi}{|J|}} \prod_{J \in \pi} |J|!^{-2} \\
&= \sum_{t=1}^k \sum_{\pi \in \Pi^\tau_t} \bX^{\tau / \pi}(\sigma) \prod_{J \in \pi} |J|!^{-2},
\end{align*}
where limits are defined by convergence in probability.
\end{proof}

We may now prove Proposition \ref{prop4inflatable}.  Define the polynomial $Y_{\tau}(m)$ to be the number of times that $\tau \in \mathfrak{S}_k$ must occur in a permutation on $m$ symbols if $\tau \otimes \rho_n$ is asymptotically $k$-symmetric.  (The above Lemma ensures that this is indeed a polynomial.)  We have, by previous computations,
\begin{align*}
Y_{(1)}(m) &= m \\
Y_{(12)}(m) = Y_{(21)}(m) &= \frac{m(m-1)}{4} \\
Y_{(123)}(m) = Y_{(321)}(m) &= \frac{m(m-1)(m-5)}{36} \\
Y_{(132)}(m) = Y_{(213)}(m) = Y_{(231)}(m) = Y_{(312)}(m) &= \frac{m(2m-1)(m-1)}{72}.
\end{align*}

\begin{proof}
We employ the follow table in our calculations.  Note that each line containing elements $\pi$ of $\Pi^{\tau}_t$ is followed by the corresponding list of permutations $\tau / \pi$.
$$
\begin{array}{c||l|l|l|l}
\tau & \Pi^{\tau}_1 & \Pi^{\tau}_2 & \Pi^{\tau}_3 & \Pi^{\tau}_4 \\
\hline
(1234) & 1234 & 1|234,12|34,123|4 & 1|2|34,1|23|4,12|3|4 & 1|2|3|4 \\
       & (1)  & (12),(12),(12)    & (123),(123),(123)    & (1234) \\
\hline
(1243) & 1234 & 1|234,12|34       & 1|2|34,12|3|4        & 1|2|3|4 \\
       & (1)  & (12),(12)         & (123),(132)          & (1243)\\
\hline
(1342) & 1234 & 1|234             & 1|23|4               & 1|2|3|4 \\
       & (1)  & (12)              & (132)                & (1342) \\
\hline
(1432) & 1234 & 1|234             & 1|2|34,1|23|4        & 1|2|3|4 \\
       & (1)  & (12)              & (132),(132)          & (1432) \\
\hline
(1324) & 1234 & 1|234,123|4       & 1|23|4               & 1|2|3|4 \\
       & (1)  & (12),(12)         & (123)                & (1324) \\
\hline
(2143) & 1234 & 12|34             & 1|2|34,12|3|4        & 1|2|3|4 \\
       & (1)  & (12)              & (213),(132)          & (2143) \\
\hline
(2413) & 1234 &                   &                      & 1|2|3|4 \\
       & (1)  &                   &                      & (2413) \\
\end{array}
$$
Then
\begin{align*}
Y_{(1234)} &= \frac{m^4}{576} - \sum_{t=1}^3 \sum_{\pi \in \Pi^{(1234)}_t} Y_{(1234) / \pi} \prod_{J \in \pi} |J|!^{-2} \\
&= \frac{m^4}{576} - \frac{Y_{(1)}}{4!^2} - \frac{2 Y_{(12)}}{1!^2 \cdot 3!^2} - \frac{Y_{(12)}}{2!^4} - \frac{3 Y_{(123)}}{1!^4 2!^2} \\
&= \frac{m^4}{576} - \frac{m}{576} - \frac{m(m-1)}{72} - \frac{m(m-1)}{64} - \frac{m(m-1)(m-5)}{48} \\
&= \frac{m^4 - 12m^3 + 55m^2 - 44m}{576} = \frac{m(m-1)(m^2-11m+44)}{576}.
\end{align*}
\begin{align*}
Y_{(1243)} &= \frac{m^4}{576} - \frac{Y_{(1)}}{4!^2} - \frac{Y_{(12)}}{1!^2 \cdot 3!^2} - \frac{Y_{(12)}}{2!^4} - \frac{Y_{(123)}}{1!^4 2!^2} - \frac{Y_{(132)}}{1!^4 2!^2} \\
&= \frac{m^4}{576} - \frac{m}{576} - \frac{m(m-1)}{144} - \frac{m(m-1)}{64} \\
& \qquad - \frac{m(m-1)(m-5)}{144}  - \frac{m(2m-1)(m-1)}{288} \\
&= \frac{m^4 - 8m^3 + 17m^2 - 10m}{576} = \frac{m(m-1)(m-2)(m-5)}{576}.
\end{align*}
\begin{align*}
Y_{(1342)} &= \frac{m^4}{576} - \frac{Y_{(1)}}{4!^2} - \frac{Y_{(12)}}{1!^2 \cdot 3!^2} - \frac{Y_{(132)}}{1!^4 2!^2} \\
&= \frac{m^4}{576} - \frac{m}{576} - \frac{m(m-1)}{144} - \frac{m(2m-1)(m-1)}{288} \\
&= \frac{m^4 - 4m^3 + 2m^2 + m}{576} = \frac{m(m-1)(m^2-3m-1)}{576}.
\end{align*}
\begin{align*}
Y_{(1432)} &= \frac{m^4}{576} - \frac{Y_{(1)}}{4!^2} - \frac{Y_{(12)}}{1!^2 \cdot 3!^2} - \frac{2 Y_{(132)}}{1!^4 2!^2} \\
&= \frac{m^4}{576} - \frac{m}{576} - \frac{m(m-1)}{144} - \frac{m(2m-1)(m-1)}{144} \\
&= \frac{m^4 - 8m^3 + 8m^2 - m}{576} = \frac{m(m-1)(m^2-7m+1)}{576}.
\end{align*}
\begin{align*}
Y_{(1324)} &= \frac{m^4}{576} - \frac{Y_{(1)}}{4!^2} - \frac{2 Y_{(12)}}{1!^2 \cdot 3!^2} - \frac{Y_{(132)}}{1!^4 2!^2} \\
&= \frac{m^4}{576} - \frac{m}{576} - \frac{m(m-1)}{72} - \frac{m(m-1)(m-5)}{144} \\
&= \frac{m^4 - 4m^3 + 16m^2 - 13m}{576} = \frac{m(m-1)(m^2-3m+13)}{576}.
\end{align*}
\begin{align*}
Y_{(2143)} &= \frac{m^4}{576} - \frac{Y_{(1)}}{4!^2} - \frac{Y_{(12)}}{2!^4} - \frac{Y_{(213)} + Y_{(132)}}{1!^4 2!^2} \\
&= \frac{m^4}{576} - \frac{m}{576} - \frac{m(m-1)}{64} - \frac{m(2m-1)(m-1)}{144} \\
&= \frac{m^4 - 8m^3 + 3m^2 +4m}{576} = \frac{m(m-1)(m^2-7m-4)}{576}.
\end{align*}
\begin{align*}
Y_{(2413)} &= \frac{m^4}{576} - \frac{Y_{(1)}}{4!^2} \\
&= \frac{m^4}{576} - \frac{m}{576} \\
&= \frac{m^4 - m}{576} = \frac{m(m-1)(m^2+m+1)}{576}.
\end{align*}
\end{proof}

\section{Tables}

\noindent Values of the function $a_p(k)$:
\begin{center}
\begin{tabular}[c]{clllllllllllllllllll}
$a_p(k)$ &2&3&4&5&6&7&8&9&10&11&12&13&14&15&16&17 \\
\hline
2  &	2&2&5&5&6&6&10&10&11&11&13&13&14&14&19&19 \\
3  &	0&2&2&2&3&3&3&6&6&6&7&7&7&8&8&8 \\
5  &	0&0&0&2&2&2&2&2&3&3&3&3&3&4&4&4 \\
7  &    0&0&0&0&0&2&2&2&2&2&2&2&3&3&3&3 \\
11 &	0&0&0&0&0&0&0&0&0&2&2&2&2&2&2&2 \\
13 &	0&0&0&0&0&0&0&0&0&0&0&2&2&2&2&2 \\
17 &	0&0&0&0&0&0&0&0&0&0&0&0&0&0&0&2
\end{tabular}
\end{center}
\vspace{.25in}

\noindent Values of the minimum modulus $T(k)$.
\begin{center}
\begin{tabular}{c|c|c|c|c}
1&2&3&4&5 \\
\hline
1& 4& 36& 288& 7200\\
\hline
\hline
6&7&8&9&10 \\
\hline
43200 & 2116800 & 33868800 & 914457600 & 9144576000
\end{tabular}
\end{center}

\noindent Values of the function $f(k)$.
\begin{center}
\begin{tabular}{c|c|c|c|c|c|c|c}
1&2&3&4&5&6&7&8 \\
\hline
1& 4& 9& 64& 128 & 352 & 1377 & 180225
\end{tabular}
\end{center}

\bibliographystyle{plain}
\bibliography{andrew}

\end{document}